\documentclass{amsart}
\usepackage{amscd,amssymb,amsmath,amsthm}
\usepackage[dvips]{graphicx}
\usepackage[all]{xy}
\usepackage{appendix}
\theoremstyle{plain}
\newtheorem{definition}{Definition}
\newtheorem{proposition}{Proposition}
\newtheorem{theorem}[proposition]{Theorem}

\newtheorem{condition}[proposition]{Condition}
\newtheorem{lemma}[proposition]{Lemma}

\newtheorem*{proposition*}{Proposition}
\newtheorem*{theorem*}{Theorem}
\newtheorem*{corollary*}{Corollary}
\newtheorem*{lemma*}{Lemma}
\newtheorem*{remark*}{Remark}
\newtheorem*{example*}{Example}

\newcommand{\Z}{\mathbb{Z}}
\newcommand{\Q}{\mathbb{Q}}

\begin{document}

\title{Open Gromov-Witten theory on Calabi-Yau three-folds I}

\author{Vito Iacovino}


\email{vito.iacovino@gmail.com}

\date{version: \today}


\begin{abstract}
We propose a general theory of the Open Gromov-Witten invariant on Calabi-Yau three-folds. 
We introduce the moduli space of multi-curves and show how it leads to invariants. Our construction is based on an idea of Witten.

In the special case that each connected component of the Lagrangian submanifold has the rational homology of a sphere we define rational numbers $F_{g,h}$ for each genus $g$ and $h$ boundary components.  
\end{abstract}

\maketitle

\section{Introduction}

Let $M$ be a Calabi-Yau three-fold and let $L$ be a special Lagrangian submanifold of $M$. 
The open Gromov-Witten invariants of the pair $(M,L)$ are supposed to count the pseudo-holomorphic bordered curves of $M$ with boundary mapped into $L$, 
or equivalently the Euler characteristic of the moduli space of pseudo-holomorphic bordered curves. 
Mathematicians have been unable to give a general construction of these invariants basically because the moduli space of pseudo holomorphic bordered curves has codimension one boundary, making its Euler characteristic ill defined. 
However a definition was given assuming particular symmetries of the pair $(M,L)$ (see \cite{L}, \cite{So}).

On the other hand, Witten in \cite{W} defines (at the physical level of rigor) the open topological string partition function for a pair $(X,L)$. 
It is well known that the closed topological string partition function can be computed in terms of closed Gromov-Witten invariants.
By analogy, the result of \cite{W} predicts the existence of Open Gromov-Witten invariants. 
These invariants have also been computed in string theory in many examples. 

In this paper we give a mathematical construction of the invariant of \cite{W} and therefore we provide a mathematical definition of Open Gromov-Witten invariants. 

There are two phenomena that make the boundary of the moduli space of pseudo holomorphic disks of codimension one: 
the bubbling off of disks and the bubbling off of spheres from a constant disks. 

The bubbling off of spheres from constant disks is less harmful since it can happen only in a discrete set points of the space of parameters 
that we need to define the moduli space of disks (such as the space of compatible almost complex structures). 
To deal with this problem we assume that $L$ is homologically trivial (we do not assume that $L$ is connected). 
This will allow us to add correction terms that compensate the bubbling off of spheres from constant disks.
This condition is also natural in string theory since it implies that the topological charge of the D-brane is zero. 
It was also used in mathematics independently by Joyce.

The problem of the splitting of the curves is a much more delicate issue since it happens along the entire space of parameters, 
making meaningless the count of the number of curves. 
In fact, the number of open curves strongly depends on the perturbation of the pseudo-holomorphic equation. 

The key to solve this problem is suggested by the result of \cite{W}. 
Witten argues that the open topological string partition function receives contributions also from  degenerate or partialy degenerate curves. 
These objects are made by usual curves joined by infinite thin strips living on $L$. 
Witten proposes that this contribution can be computed in terms of Chern-Simons integrals. 
For example, in the case of the cotangent bundle $M=T^*L$, where there are no non-constant holomorphic curves with boundary mapped into the zero section, 
the open topological string is equivalent to Chern-Simons theory on $L$. The degenerate curves correspond to Feynman graphs of the Chern-Simons theory. 
For more general Calabi-Yau manifolds, where pseudo-holomorphic curves can be present, 
Witten computes the contribution of degenerate curves using Wilson loop integrals associated to the boundary of the curves (see formula $(4.50)$ of \cite{W})). 


From the work of Witten it follows that open topological string theory does not simply count open curves, but it counts what we define as multi-curves, 
in analogy with the physical terminology multi-instantons. 
Roughly, a multi-disk is described by a tree of disks. 
The moduli space of multi-curves provides the natural mathematical setting for the Feynman expansion of the Lagrangian $(4.50)$ of \cite{W} 
and therefore for the general solution of the problem of splitting of curves. 
We give a mathematical regularization of the Feynman expansion of the Lagrangian $(4.50)$ of \cite{W}. 

The usual virtual computation applied to the space of multi-curves leads to a new algebraic topological object. 
We call this object system of singular chains. 
This object encodes all the informations for the curves counting.

In the particular case that $L$ has the rational homology of a sphere, 
we apply this general formalism to show that it naturally leads to Open Gromov-Witten invariants. 
The contribution of multi-disks to the Open Gromov-Witten invariant is computed as generalization of linking numbers of the boundary components of the disk vertices. 
Let us explain in this particular case how the problem of the bubbling off of disks is solved. 
Consider for example a one parameter family of compatible almost complex structures of $M$. Suppose that a disk $A$ splits into two disks $B$ and $C$. 
The contribution of the $2$-disk $B,C$ is given by the linking number of the boundary knots. 
When the disk $A$ disappears the boundary knots of the disks $B$ and $C$ cross and therefore the linking number of $B,C$ jumps. 
Similarly, disappearing of an $n$-disk is compensated by the discontinuity in the $m$-disk contribution for $m>n$.    

In \cite{frame}, we apply the analysis of this paper to non-compact situations. This includes a mathematical definition of the invariants arising in Large $N$ duality.   
   
In the sequel \cite{I1} we consider the case that $H_1(L,\Q) \neq 0$. 
We construct the Open Gromov-Witten potential $S$ as a homotopy class of solutions of the Master equation in the ring of functions on $H^*(L, \Q)$ with coefficients in the Novikov ring. 
This is the effective action of the Open Topological String of \cite{W}. 
$S$ is defined up to master homotopy, unique up to equivalence. The evaluation of $S$ on its critical points leads to numerical invariants.

The paper is organized as follows. 
In section two we define the Kuranishi space of multi-disks. 
The orientation of the moduli space of disks induces an orientation on the moduli space of multi-disks. 
We show how the boundary faces of the space of multi-disks can be identified with some subspaces of different components.

In section three we define the notion of systems of chains. This is the main algebraic topological object in which we will be interested.
 
In section four we introduce the constrain on the perturbation of the Kuranishi structure of multi-disks. 
With this constrain, we show that the evaluation on the punctures of multi-disks leads to a system of chains well defined up to homotopy. 
This procedure is a minor adaptation of the procedure developed in \cite{FO}. 

In section five we discuss the gluing property. This property assure that the system of chains is compatible with the cut of a multi-disk along an internal edge. 

In section five we specialize to the case that each connected component of $L$ has the rational homology of a sphere. 

In section six we consider the higher genus case. 
Most of what we have said for multi-disks can be adapted to multi-curves. 
In the case that each connected component of $L$ has the rational homology of a sphere 
we define the Open Gromov-Witten invariants $F_{g,h} \in \Q$ counting multi-curves of genus $g$ and $h$ boundary components.

{\bf Acknowledgements. } We are grateful to D. Joyce for many useful correspondences. We also thank B. Siebert and J. Latchev for their comments on a preview version.

\section{Multi-disks}

Let $M$ be a Calabi-Yau three-fold and let $L$ a Maslov index zero Lagrangian submanifold of $M$. 
In this section we define the moduli space of multi-disks with boundary mapped into $L$. 

Given a ribbon tree $T$ we denote by $V(T)$ the set of vertices of $T$ and by $E(T)$ the set of edges of $T$. 
$H(T)$ is the set of edges of $T$ with an assigned orientation, or equivalently the set of half-edges of $T$. 
Thus for each internal edge there correspond two elements of $H(T)$, and for each external edge there is one element of $H(T)$, the oriented edge starting in the unique vertex. 
For $v \in V(T)$, $H(v)$ denotes the set of oriented edges starting from $v$.

\begin{definition} A decorated tree is a ribbon tree $T$ with a relative homology class $A_v \in H_2(M,L)$ assigned to each vertex $v$. 
If $A_v \neq 0$ the vertex $v$ can have arbitrary valence. If $A_v =0$ the valence of $v$ has to be at least $3$. 

The relative homology class of $T$ is by definition the sum of the homology classes of its vertices $A = \sum_{v \in V(T)} A_v$.
We denote by $\mathcal{T}$ the set of decorated trees and by $\mathcal{T}(A)$ the set of decorated trees in the homology class $A$.
\end{definition}

Fix a compatible almost-complex structure $J$ on $M$. 
Let $ \overline{\mathcal{M}}_n(A)$ be the moduli space of pseudo-holomorphic disks with $n$ marked points with a fixed cyclic order in the relative homology class $A \in H_2(X,L)$ 
(this is usualy denoted the main component in \cite{FO3}). 
By Chapter $7$ of \cite{FO3}, $\overline{\mathcal{M}}_n(A)$ can be endowed with a weakly submersive Kuranishi structure with corners.

For each $v \in V(T)$ denote by $\overline{\mathcal{M}}_v$ a copy of $\overline{\mathcal{M}}_{H(v)}(A_v)$ 
where we consider the boundary marked points labeled with the half edges $H(v)$ starting in $v$, respecting the cyclic order of the ribbon structure.

The moduli space of $T$-multi-disks $ \overline{\mathcal{M}}_T$ is given by
$$ \overline{\mathcal{M}}_T= \left( \prod_{v \in V(T)} \overline{\mathcal{M}}_v \right) / \text{Aut}(T). $$
The space $\overline{\mathcal{M}}_T$ has a natural Kuranishi structure. Moreover this Kuranishi structure has a tangent space.

\subsection{Boundary}

Fix a decorated tree $T$. For each $v \in V(T)$ the boundary of $\overline{\mathcal{M}}_v$ can be decomposed into different components 
\begin{equation} \label{bubbling}
\partial (\overline{\mathcal{M}}_{H(v)} (A_v) ) = \left( \bigsqcup_{\substack{H_1 \cup H_2=H(v) \\ A_1+A_2=A }} \overline{\mathcal{M}}_{\{H_1 \cup * \}} (A_1) \times_L 
\overline { \mathcal{M}}_{ \{ H_2 \cup * \}}(A_2) \right) / \Z_2 .
\end{equation}
Here the union is taken over the decompositions $H= H_1 \sqcup H_2  $ preserving the cyclic order. The marked point $*$ is the singular point associated to the boundary face. 
The fiber product in the left is defined using the evaluation map in $*$.  

In the case that $H(v)$ is empty we need to add an extra term to (\ref{bubbling}):
\begin{equation} \label{bubbling0}
\partial \overline{\mathcal{M}}_0 = (\overline{\mathcal{M}}_1 \times_L \overline{\mathcal{M}}_1) / \Z_2 \sqcup \overline{\mathcal{M}}_{0,1} \times_M L . 
\end{equation} 
Here $\overline{\mathcal{M}}_{0,1}$ is the moduli space of $J$-pseudo-holomorphic spheres with one marked point.
The last term of (\ref{bubbling0}) comes from the bubbling of the spheres from constant disks. It has been discussed in section $7.4$ of \cite{FO3}. 
These are the boundary nodes of type E in Definition $3.4$ of \cite{L}.

For each $v \in V(T)$ define
\begin{equation} \label{bvertex}
\partial_v \overline{\mathcal{M}}_T =  \left( \partial  \overline{\mathcal{M}} _v \times  \prod_{v' \neq v} \overline{\mathcal{M}}_{v'}  \right) / \text{Aut}(T,v) .
\end{equation}
The boundary of $\overline{\mathcal{M}}_T  $ can be subdivided into the Kuranishi spaces $\partial_v \overline{\mathcal{M}}_T $
$$  \partial \overline{\mathcal{M}}_T = \bigsqcup_{v \in V(T)} \partial_v \overline{\mathcal{M}}_T .$$

For each $e \in E(T)$ internal edge let
\begin{equation} \label{bedge}
\partial_e \overline{\mathcal{M}}_T  = \left( (\text{ev}_e)^{-1}(\Delta) \right) /  \text{Aut}(T,e)   
\end{equation}
where $\Delta \subset L \times L$ is the diagonal and $ \text{ev}_e : \prod_{v \in V(T)} \overline{\mathcal{M}}_v \rightarrow  L \times L $ is the evaluation on the punctures associated to $e$. 
Observe that $\text{ev}_e$ is a submersion since the Kuranishi structure on the disks was taken weakly submersive. 
This implies that $\partial_e \overline{\mathcal{M}}_T $ has a natural Kuranishi structure. 

The definition (\ref{bedge}) can be generalized to more edges. For $e_1,e_2$ internal edges of $T$ we define 
$$  \partial_{e_1, e_2} \overline{\mathcal{M}}_T  = \left( (\text{ev}_{e_1})^{-1}(\Delta) \cap (\text{ev}_{e_2})^{-1}(\Delta) \right) /  \text{Aut}(T,e_1,e_2)  .$$
This is a Kuanishi space again because the Kuranishi structures on the disks was taken weakly submersive. 
We also have 
$$\partial_{e_1, e_2} \overline{\mathcal{M}}_T = \partial_{e_1} \overline{\mathcal{M}}_T  \sqcap \partial_{e_2} \overline{\mathcal{M}}_T .$$   
The intersections between more of spaces (\ref{bedge}) can be discussed in a similar way.

\begin{definition} Let $e \in E(T)$ be an internal edge attached to the vertices $v_1, v_2$. 
The decorated tree $T/e$ is the tree obtained by contracting $e$ to a vertex $e/e$. 
The relative homology class of the vertex $e/e$ is the sum of the homology classes of the vertices $v_1$ and $v_2$, $ A_{e/e}  = A_{v_1} + A_{v_2} $.
\end{definition}
 
The Kuranishi space $\partial_v \overline{\mathcal{M}}_T$ is subdivided in boundary faces of $ \overline{\mathcal{M}}_T$. In fact we have:  
\begin{lemma}  \label{attach-boundary}
If $T$ has at least one edge, there is an isomorphism of Kuranishi spaces:
\begin{equation} \label{attach}
\partial_v \overline{\mathcal{M}}_T  \cong \bigsqcup_{(T'/e',e'/e')=(T,v)} \partial_{e'} \overline{\mathcal{M}}_{T'}
\end{equation} 
where the union is over all the trees $T'$ and edges $e' \in E(T')$ such that $(T'/e',e'/e')=(T,v)$.

Let $T_0$ be the tree with no edges and one vertex $v$. There is an isomorphism of Kuranishi spaces:
\begin{equation} \label{attach0}
\partial_v \overline{\mathcal{M}}_{T_0}=  \partial_e \overline{\mathcal{M}}_{T_1}  \sqcup \overline{\mathcal{M}}_{0,1} \times_M L 
\end{equation}
where $T_1$ is the graph given by two vertices connected by an edge $e$. 


\end{lemma}

\begin{proof}
The lemma is immediate from the definition. Consider first (\ref{attach0}) and let $T_1$ be the graph given by two vertices $w,z$ connected by an edge $e$. 
By definition we have
$$ \partial_{e} \overline{\mathcal{M}}_{T_1}= ((\overline{\mathcal{M}}_w \times \overline{\mathcal{M}}_z) \times_{L^2} \Delta)  / \Z_2 .$$
and equation (\ref{bubbling0})
Where the fiber products are made using the evaluation map as usual. Therefore formula (\ref{attach0}) follows from  
$$  (\mathcal{M}_1 \times_L \mathcal{M}_1) = (\mathcal{M}_1 \times \mathcal{M}_1) \times_{L^2} \Delta $$ 
In the same way we can prove (\ref{attach}) using (\ref{bubbling}).
\end{proof}

In order to deal with the last term in relation (\ref{attach0}) we need to consider an extra Kuranishi space. 
We \emph{assume that the homology class of $L$ in $M$ is trivial}. Let $B \in C_4(M)$ be a singular chain such that $\partial B =L$. 
For a fixed $B$ we can consider the fiber product
\begin{equation} \label{alternative}
\overline{\mathcal{M}}_{0,1} \times_M B 
\end{equation}
as Kuranishi space.
The boundary of the space (\ref{alternative}) can be identified with the last term of relation (\ref{attach0})
\begin{equation} \label{spherebubbling}
\partial (\overline{\mathcal{M}}_{0,1} \times_M B ) \cong \overline{\mathcal{M}}_{0,1} \times_M L .
\end{equation}

\subsection{Orientation}


Let $T \in \mathcal{T}$ be a decorated tree.
Define
$$ \mathfrak{o}_T = ( \otimes_{e \in E(T)} \mathfrak{o}_e ) \otimes \mathfrak{o}_{\text{ex}} $$
where
$$ \mathfrak{o}_e = \{ \text{set of orientations of the edge $e$} \} \cong \Z_2 $$ 
and
$$\mathfrak{o}_{\text{ex}} = \{ \text{set of parities of the ordering of the external edges of $T$} \} \cong \Z_2 . $$ 


For each vertex $v \in V(T)$, let $\mathcal{M}_{0,v}$ be a copy of the moduli space of disks without punctures in the homology class $A_v$.



 
We can identify $\mathcal{M}_T$ with an open subset of 
$$ \prod_{v \in V(T)} \mathcal{M}_{0,v} \times \prod_{e \in E(T)} \partial D_e  $$
where $\partial D_e$ is the boundary of the disk from where the edge $e$ starts.  
It follows that an orientation of the moduli space of disks induces an orientation of $\mathcal{M}_T$ with twisted coefficients $\mathfrak{o}_T$

Now, \emph{assume that $L$ is oriented and spin}. 
By Section $44$ of \cite{FO3}, $\mathcal{M}_{0,v}$ has a natural orientation. Therefore we have:
\begin{lemma}
There exists a natural orientation of $\mathcal{M}_T$ with twisted coefficients $\mathfrak{o}_T$. 
\end{lemma}

Now we want to study how the isomorphisms of Lemma \ref{attach-boundary} are related to this orientation.  
\begin{lemma}
For each internal edge $e$, the Kuranishi space $\partial_e \overline{\mathcal{M}}_T  $ have a natural orientation with twisted coefficients in $\mathfrak{o}_{T/e}$.
\end{lemma}
\begin{proof}
By definition (\ref{bedge}) an orientation of the normal bundle $N_*\Delta$ of $\Delta$ in $L \times L$ induces an orientation of $\partial_e \overline{\mathcal{M}}_T  $ with twisted coefficients $\mathfrak{o}_T$ 
according to the isomorphism
$$ T_*(\overline{\mathcal{M}}_T) = T_*(\partial_e \overline{\mathcal{M}}_T)  \times  N_*\Delta .$$
Since $L$ is oriented, an orientation of $e$ induces an orientation of $N_*\Delta$. In other words $N_*\Delta$ has a natural orientation with twisted coefficients $\mathfrak{o}_e$.
The Lemma follows.
\end{proof}

Endow $\partial_v \overline{\mathcal{M}}_T $ with the orientation induced as boundary of $\mathcal{M}_T$.
\begin{lemma} \label{opposite}
(\ref{attach}) and (\ref{attach0}) are isomorphisms of oriented Kuranishi spaces.
\end{lemma}
\begin{proof}
Consider first the case that the graph $T=T_0$ is the graph given by just one vertex $v$. $T'=T_1$ is given by two vertices connected by an edge $e$. 

From Lemma \ref{compare-orientation} with $X_1 = X_2 = \overline{\mathcal{M}}_1$,  $B=L$ and $f_1=f_2$ be the evaluation map in the boundary marked point, we have
$$  \overline{\mathcal{M}}_1 \times_L \overline{\mathcal{M}}_1=  ( \overline{\mathcal{M}}_1 \times \overline{\mathcal{M}}_1 )\cap (\text{ev})^{-1}(\Delta) $$
as oriented Kuranishi spaces.
By definition we have that 
$$\partial_e \overline{\mathcal{M}}_{T_1} =  (( \overline{\mathcal{M}}_1 \times \overline{\mathcal{M}}_1 )\cap (\text{ev})^{-1}(\Delta) )/ \Z_2 $$
and by Proposition $8.3.3$ of \cite{FO3} we have that 
$$\partial_v \overline{\mathcal{M}}_{T_0} = ( \overline{\mathcal{M}}_1 \times_L \overline{\mathcal{M}}_1 )/ \Z_2 . $$
as oriented Kuranishi spaces.

The proof for general trees can be reconducted to this case.
\end{proof}

\begin{lemma} \label{compare-orientation}
Let $X_1$ and $X_2$ be oriented Kuranishi spaces. $f_1 : X_1 \rightarrow B$ and $f_2 : X_2 \rightarrow B$ weakly submersive. 
There exists an isomorphism of oriented Kuranishi spaces
$$   X_1 \times X_2 \cap (f_1 \times f_2)^{-1}(\Delta) =   (-1)^{\text{dim} B (\text{dim}X_1 -  \text{dim} B)} X_1 \times_B X_2$$
where $\Delta \subset B \times B$ is the diagonal.
\end{lemma}
\begin{proof}
We use the notation of Paragraph $8.2$ of \cite{FO3}. For $i=1,2$, $X_i$ is a space with Kuranishi structure $(s_i; E_i \rightarrow U_i)$. 
By definition (see Paragraph $8.2$ of \cite{FO3}) we have 
$$T_*(U_1 \times U_2) = (-1)^{\text{rank} E_2 (\text{dim}U_1 -  \text{rank} E_1)} (E_1 \oplus E_2)_* \times T_*(X_1 \times X_2)$$
$$T_*(U_1 \times_B U_2) = (-1)^{\text{rank} E_2 (\text{dim}X_1 -  \text{dim} B)} (E_1 \oplus E_2)_* \times T_*(X_1 \times_B X_2) .$$
From the relations 
\begin{eqnarray*}
TU_1 \times TU_2 &= TU_1^{\circ} \times T_*B \times T_*B  \times^{\circ} TU_2 = TU_1^{\circ} \times T_* \Delta \times N_*\Delta  \times^{\circ} TU_2= \\
& (-1)^{\text{dim} B (\text{dim}U_2 -  \text{dim} B)}  TU_1^{\circ} \times T_* \Delta \times^{\circ} TU_2  \times N_*\Delta = \\ 
& (-1)^{\text{dim} B (\text{dim}U_2 -  \text{dim} B)}  T_*(U_1 \times_B U_2)  \times N_*\Delta .
\end{eqnarray*}
it follows that
$$  T_*(X_1 \times X_2) =  (-1)^{\text{dim} B (\text{dim}X_1 -  \text{dim} B)}  T_*(X_1 \times_B X_2) \times   N_*\Delta .$$
\end{proof}

\section{Systems of singular chains}
In this section we introduce the algebraic topological tools necessary for the defition of open Gromov-Witten invariants.

For each decorated tree $T$, define $L_T$ as the orbifold
$$ L_T = L^{H(T)} / \text{Aut}(T) .$$
For each internal edge $e \in E(T)$, denote by $\Delta_e$ the big diagonal corresponding to $e$, 
that is the preimage of the diagonal of $L \times L $ under the projection $ L^{H(T)} \rightarrow L^2 $ associated to the edge $e$. 

\begin{definition} \label{def-system}
A system of singular chains 
$$W_{\mathcal{T}} = \{  W_T  \}_{T \in \mathcal{T}} $$ 
is a collection of singular chains
$W_T \in C_{|E(T)|}( L_T , \mathfrak{o}_T) $
with twisted coefficients in $\mathfrak{o}_T$.   

Fulfilling the following properties:
\begin{itemize}
\item[(a)] For each $T \in \mathcal{T}$, and $e$ an internal edge of $T$, $W_T$ intersects $\Delta_e$ transversely. 

Let $\partial_e W_T$ be the singular chain on $L_{T/e}$ defined by  
$$ \partial_e W_T = W_T \cap \Delta_e \in C_*(L_{T/e},\mathfrak{o}_{T/e} ) .$$ 

\item[(b)] The following identity holds as currents  
\begin{equation} \label{boundary-collection}
\partial W_T = \sum_{T'/e=T} \partial_{e'} W_{T'} .
\end{equation}
Here the sum is taken over all the pairs $(T',e')$ with $e' \in E(T')$ such that $T'/e' \cong T$.
\end{itemize}
\end{definition}


We also need to consider homotopies of systems of singular chains and equivalences of homotopies.

\begin{definition} A homotopy $Y_{\mathcal{T}}$ between $W_{\mathcal{T}}$ and $W_{\mathcal{T}}'$ is a collection of singular chains $\{ Y_T \}_{T \in \mathcal{T}}$
$$Y_T \in C_{|E(T)|+1}([0,1] \times L_T , \mathfrak{o}_T )$$
fulfilling the following properties:
\begin{itemize}
\item[(a)] For each $T \in \mathcal{T}$, and $e$ an internal edge of $T$, $Y_T$ intersects transversely $[0,1] \times \Delta_e$. 
In particular $ Y_T \cap [0,1] \times \Delta_e$ defines a singular chain on $[0,1] \times L_{T/e}$ 
$$ \partial_e Y_T = Y_T \cap ([0,1] \times \Delta_e ) \in C_*([0,1] \times L_{T/e}, \mathfrak{o}_{T/e}) .$$ 
\item[(b)] The following identity holds as currents  
\begin{equation} \label{boundary-homotopy}
\partial Y_T = \sum_{T'/e'=T} \partial_{e'} Y_{T'} + \{ 0 \} \times W_T - \{ 1 \} \times W_T'
\end{equation}
where the sum is over all the pairs $(T',e')$ with $e' \in E(T')$ and such that $T'/e' \cong T$.
\end{itemize}

Two homotopies $Y_{\mathcal{T}}$ and $X_{\mathcal{T}}$ between $W_{\mathcal{T}}$ and $W_{\mathcal{T}}'$ are called equivalent if there exists a collection of singular chains $\{ Z_T \}_{T \in \mathcal{T}}$
$$Z_T \in C_{|E(T)|+2}([0,1]^2 \times L_T ,\mathfrak{o}_T )$$
fulfilling the following properties:
\begin{itemize}
\item[(a)] For each $T \in \mathcal{T}$, and $e$ an internal edge of $T$, $Z_T$ intersects $[0,1]^2 \times \Delta_e$ transversely. 
In particular $ Z_T \cap ([0,1]^2 \times \Delta_e)$ defines a singular chain on $[0,1]^2 \times L_{T/e}$ 
$$ \partial_e Z_T = Z_T \cap ([0,1]^2 \times \Delta_e) \in C_*([0,1]^2 \times L_{T/e},  \mathfrak{o}_{T/e}) .$$ 
\item[(b)] The following identity holds as currents  
\begin{equation} \label{boundary-equivalence}
\partial Z_T = \sum_{T'/e'=T} \partial_{e'} Z_{T'} + [0,1] \times \{ 0 \} \times W_T - [0,1] \times \{ 1 \} \times W_T' + \{ 0 \} \times Y_T - \{ 1 \} \times X_T
\end{equation}
where the sum is over all the pairs $(T',e')$ with $e' \in E(T')$ and such that $T'/e' \cong T$.
\end{itemize}
\end{definition}

\section{Invariance}

\subsection{Perturbation of the Kuranishi structure}

Let $\mathfrak{s}$ be a perturbation of the Kuranishi structure $s$ of the spaces 
$ \{ \overline{\mathcal{M}}_T \}_{T \in \mathcal{T}}$ and $\overline{\mathcal{M}}_{0,1} \times_M B $.

We consider perturbations fulfilling the following condiction 
\begin{condition} \label{compatibility}
$\mathfrak{s}$ is transversal to the zero section in each $\overline{\mathcal{M}}_T$ and on each stratum of the boundary. 
$\mathfrak{s}$ is compatible with the isomorphism (\ref{attach}), that is
$$ \bigsqcup_{T',e'} \mathfrak{s} |_{\partial_{e'} \overline{\mathcal{M}}_{T'}} = \mathfrak{s} |_{\partial \overline{\mathcal{M}}_{T}}  $$
for each decorated tree $T$.
In the special case that $T=T_0$ is the tree with no edges, we impose
$$  \mathfrak{s} |_{\partial_{e} \overline{\mathcal{M}}_{T_1}} \sqcup \mathfrak{s} |_{\partial (\overline{\mathcal{M}}_{0,1} \times_M B )} =
 \mathfrak{s} |_{\partial \overline{\mathcal{M}}_{T_0}}  .$$
according to isomorphisms (\ref{attach0}) and (\ref{spherebubbling}),

\end{condition}

Using the standard machinery developed in Section $6$ of \cite{FO} or Appendix A of \cite{FO3} 
it is possible to prove that there exist a perturbation $\mathfrak{s}$ satisfying condition \ref{compatibility} 
as close as we want in the $C^0$ topology to the starting $s$ (see Lemma \ref{trasverse} below). 

Let $\mathfrak{s}$ be any such perturbation. 

There exists a natural strongly continuous map 
\begin{equation} \label{strong}
\text{ev}:\overline{\mathcal{M}}_T(J) \rightarrow L_T .
\end{equation}
As in formula $(6.10)$ of \cite{FO}, the strongly continuous map (\ref{strong}) defines a singular chain in $L_T$  
\begin{equation} \label{chain}
W_T = \text{ev}_* ((\mathfrak{s}|_{\overline{\mathcal{M}}_T}^{-1})(0))
\end{equation}
with twisted coefficients in $\mathfrak{o}_T$. 

If $T=T_0$ is the tree with one vertex and no edge there is an extra term in the formula (\ref{chain}): 
\begin{equation} \label{chain0}
W_{T_0} =\text{ev}_* ((\mathfrak{s}|_{ \overline{\mathcal{M}}_{T_0}}^{-1})(0)) +  \text{ev}_* ((\mathfrak{s}|_{\overline{\mathcal{M}}_{0,1} \times_M B}^{-1})(0)) \in \Q.
\end{equation}
The target of the map $\text{ev}$ in formula (\ref{chain0}) is a point. 

\begin{proposition}
Formulas (\ref{chain}) and (\ref{chain0}) define a system of chains.
\end{proposition}
\begin{proof}
Condition $(a)$ of Definition \ref{def-system} follows from the trasversality of $\mathfrak{s}$ along the spaces $\partial_{e} \overline{\mathcal{M}}_{T}$. 
Condition $(b)$ follows from Lemma \ref{opposite} and the following identities:
$$ \partial W_T =  \text{ev}_*(\mathfrak{s} |_{\partial \overline{\mathcal{M}}_{T}}^{-1}(0)) $$
$$ \partial_e W_T = \text{ev}_*(\mathfrak{s} |_{\partial_{e} \overline{\mathcal{M}}_{T}}^{-1}(0)) .$$
\end{proof}
\subsection{Invariance}

The Kuranishi structure we constructed depends on the various choices we made. However we have the following (the proof is minor adaptation of the proof of Theorem $17.11$ of \cite{FO}):  

\begin{theorem} \label{invariance}
The system of singular chains $W_{\mathcal{T}}$ depends on the almost complex structure $J$ and various choice we made to define a Kuranishi structure. 
Different choices lead to systems of singular chains that are homotopic, with homotopy determined up to equivalence.
\end{theorem}

\begin{proof} Let $J$ and $J'$ be two different complex structures compatible with the symplectic structure $\omega$. Let $J_s$ be a family of compatible almost complex structures such that $J_s=J$ for $s \in [0,\varepsilon]$ and $J_s=J'$ for $s \in [1-\varepsilon,1]$. Define
$$\overline{\mathcal{M}}_{\mathcal{T}}(J_{\text{para}}) = \cup_{s \in [0,1]} \{ s \} \times \overline{\mathcal{M}}_{\mathcal{T}}(J_s) .$$
As in Theorem $17.11$ of \cite{FO} we can endow $\overline{\mathcal{M}}_{\mathcal{T}}(J_{\text{para}})$ with a compact and Hausdorff topology. Moreover there exists a Kuranishi structure on $\overline{\mathcal{M}}_{\mathcal{T}}(J_{\text{para}})$ that extends the Kuranishi structure of $\overline{\mathcal{M}}_T(J) $ and $\overline{\mathcal{M}}_T(J') $. 

Lemma \ref{trasverse} implies that there exist a transverse multisection $\mathfrak{s}$ as close as we want to $s$ in the $C^0$ topology, satisfying Condition \ref{compatibility}. 
Moreover $\mathfrak{s}$ can be taken to extend assigned perturbations of $\overline{\mathcal{M}}_{\mathcal{T}}(J) $ and $\overline{\mathcal{M}}_{\mathcal{T}}(J') $.

As in formulas  (\ref{chain}) and (\ref{chain0}), $\mathfrak{s}$ defines a singular chain $Y_T $ in $[0,1] \times L_T $ for each tree $T$.
Lemma $4.7$ of \cite{FO} implies that $\partial Y_T$ is the intersection of $Y_T$ with the boundary of $\overline{\mathcal{M}}_T(J_{\text{para}}) $:
$$ \partial Y_T = \text{ev}_* ((\mathfrak{s}|_{\partial \overline{\mathcal{M}}_T(J_{\text{para}} )}^{-1})(0))$$
The system of chains $Y_{\mathcal{T}} = \{ Y_T  \}_{T \in \mathcal{T}}$ defines a homotopy between the system of chains $ W_{\mathcal{T}} $ and $ W_{\mathcal{T}}' $. 

\end{proof}  

\begin{lemma} \label{trasverse}
There exists a perturbation $\mathfrak{s}$ of the Kuranishi structure of $\overline{\mathcal{M}}_T(J ) $ satisfying condition \ref{compatibility}.
\end{lemma}
\begin{proof}
The construction of the multi-sections is done using the standard machinery of Kuranishi structures developed in \cite{FO} or Appendix A of \cite{FO3}.  

We will proceed by induction on the number of internal edges of the tree. In each step we impose condition \ref{compatibility}. 


Suppose first that $T$ is a tree with no internal edges (that is a vertex with some external edge). 
In this case $\overline{\mathcal{M}}_T$ is the usual space of disks with boundary punctures. Here we can just pick a multi-section $\mathfrak{s}$ transverse to the zero section. 

Now consider a tree $T$ with at least one internal edge and assume that  $\mathfrak{s}$ has been constructed for all the trees with fewer internal edges than $T$. 
For each internal edge $e \in E(T)$, condition \ref{compatibility} defines a multi-section on $ \partial_e \overline{\mathcal{M}}_T$ in terms of the multi-section of $ \partial_v \overline{\mathcal{M}}_{T/e}$ (where $v$ is the vertex of $T/e$ corresponding to the edge $e$).
By induction these multi-sections are compatible on the intersection. 
In fact (for $e_1, e_2 \in E(T)$ be two internal edges) the multi-section on $ \partial_{e_1} \overline{\mathcal{M}}_T$ and on $ \partial_{e_2} \overline{\mathcal{M}}_T$ restricted to $ \partial_{e_1,e_2} \overline{\mathcal{M}}_T$ are both defined from the multi-section on $ \partial_{v_1,v_2} \overline{\mathcal{M}}_{T/ \{e_1,e_2 \}}$ (a codimension two corner). 


We have therefore defined a multi-section on $\sqcup_e \partial_e \overline{\mathcal{M}}_T$. 
Lemma 17.4 of \cite{FO} or Theorem A$1.23$ of \cite{FO3} can be used to extend this multi-section to a transverse multi-section of $\overline{\mathcal{M}}_T $.
\end{proof}

\section{Gluing property}

Let $\mathcal{T}^{k,l}$ be the set of decorated trees with $k$ internal marked edges and $l$ external marked edges. We denote also $ \mathcal{T}^k  = \mathcal{T}^{k,0}$.
Observe that
\begin{equation} \label{cut}
\mathcal{T}^1 = (\mathcal{T}^{0,1} \times \mathcal{T}^{0,1}) / \Z_2 
\end{equation}
To the element $(T,e_0) \in \mathcal{T}^1$ corresponds the pair $(T_1,e_1), (T_2,e_2) \in \mathcal{T}^{0,1}$, 
where $T_1$ and $T_2$ are the trees made cutting the edge $e$ in two edges $e_1$ and $e_2$.

\begin{definition} \label{def-system1}
A system of singular chains with one marked edge is a collection of singular chains
$$W_{\mathcal{T}^1} = \{  W_{(T,e_0)}  \}_{(T,e_0)\in \mathcal{T}^1} $$ 
with
$W_T \in C_{|E(T)|}( L_T , \mathfrak{o}_T) $ 
with the following properties:
\begin{itemize}
\item[(a)] For each $(T,e_0) \in \mathcal{T}^1$, and $e$ internal edge of $T$ with $e \neq e_0$, $W_T$ intersects transversely $\Delta_e$. 
In particular $  W_{(T,e_0)} \cap \Delta_e$ defines a singular chain on $L_{T/e}$ 
$$ \partial_e  W_{(T,e_0)} = W_{(T,e_0)} \cap \Delta_e \in C_*(L_{T/e},\mathfrak{o}_{T/e}) .$$ 
\item[(b)] The following identity holds as currents  
\begin{equation} \label{boundary-collection1}
\partial W_{(T,e_0)} = \sum_{(T'/e',e_0')=(T,e_0)} \partial_{e'} W_{(T',e_0')} .
\end{equation}
Here the sum is over all the triples $T',e_0',e'$, with $e'$ internal edge of $T'$ such that $e' \neq e_0'$ and $(T'/e',e_0') \cong (T,e_0)$.
\end{itemize}
\end{definition}

The definitions of homotopy of chains and equivalence of homotopy can be extended in the same way. 

Observe that using the maps $\mathcal{T}^1 \rightarrow \mathcal{T}$ and $\mathcal{T}^{0,1} \rightarrow \mathcal{T}$ that forget the marked edge, 
a system of chains $W_{\mathcal{T}}$ can be lifted in a natural way to a system of chains $W_{\mathcal{T}^1}$ and $W_{\mathcal{T}^{0,1}}$.

Now, we would like define the notion of gluing property of a system of chain $W_{\mathcal{T}}$. In principle we would like that the identity 
\begin{equation} \label{strong-gluing}
W_{\mathcal{T}^1} = W_{\mathcal{T}^{0,1}} \times W_{\mathcal{T}^{0,1}} 
\end{equation}
holds. However it is easy to see that (\ref{strong-gluing}) is inconsistent with the transversality along diagonals of a system of chains. 
Instead now we define the notion of gluing property up to homotopy and prove that it holds for the system of chains we are interested in. 

We replace formula (\ref{strong-gluing}) with the weaker assumption that there exists a homotopy $Y_{\mathcal{T}^1}$ between the right and left side of (\ref{strong-gluing}). 
As usual we want that $Y_{\mathcal{T}^1}$ is uniquely determinated up to homotopy.
We also need to assure that if we made two cuts in the tree, the result is independent of the order of the cut. More precisely we have: 
\begin{definition} \label{gluing-hom}
$W_{\mathcal{T}}$ has the the gluing property up to homotopy if it is assigned an equivalence class of homotopies $Y_{\mathcal{T}^1}$ between $W_{\mathcal{T}^1}$ and $W_{\mathcal{T}^{0,1}} \times W_{\mathcal{T}^{0,1}}$ 
such that the homotopy between $W_{\mathcal{T}^2}$ and $W_{\mathcal{T}^{0,1}} \times W_{\mathcal{T}^{0,2}} \times W_{\mathcal{T}^{0,1}}$
\begin{equation} \label{YY-comp}
Y_{\mathcal{T}^2} \circ (W_{\mathcal{T}^{0,1}} \times Y_{\mathcal{T}^{1,1}})
\end{equation}
is invariant (up to equivalence) under the $\Z_2$ action that switch order of the marked edges.
\end{definition}
In (\ref{YY-comp}), $\mathcal{T}^2$ is the set of decorated trees with two marked ordered internal edges, 
$Y_{\mathcal{T}^2}$ is the lift of $Y_{\mathcal{T}^1}$ to an homotopy on $\mathcal{T}^2$, 
that is an homotopy between $W_{\mathcal{T}^2}$ and $W_{\mathcal{T}^{0,1}} \times W_{\mathcal{T}^{1,1}} $.

Let $W_{\mathcal{T}}$ be a system of chains defined as in formula (\ref{chain}) . 
Recall that two different perturbations lead to homotopic systems of chains with homotopy determined up to equivalence. 

\begin{proposition} \label{gluing-lemma}
$W_{\mathcal{T}}$ has the gluing property up to homotopy.
\end{proposition}
\begin{proof}

Let $ \mathfrak{s}_1 $ and $\mathfrak{s}_{0,1}$ be the lifting of $ \mathfrak{s}$ as pertubation of $\overline{\mathcal{M}}_{\mathcal{T}^1}(J) $ and 
$\overline{\mathcal{M}}_{\mathcal{T}^{0,1}}(J) $ respectively. 
Because of (\ref{cut}), there is an isomorphism
\begin{equation} \label{cutedge3}
\overline{\mathcal{M}}_{\mathcal{T}^1}  \cong ( \overline{\mathcal{M}}_{\mathcal{T}^{0,1}} \times \overline{\mathcal{M}}_{\mathcal{T}^{0,1}} ) / \Z_2
\end{equation}
 

There is also an obvious way to extend condition \ref{alternative} to $\mathcal{T}^1$. We do not impose the condition associated to the marked edge.  

There exists a pertubation $\tilde{\mathfrak{s}}_1$ of $\overline{\mathcal{M}}_{\mathcal{T}^1} \times [0,1]$ that satisfies (the one parameter version) of this condition and 
such that restricted to $\overline{\mathcal{M}}_{\mathcal{T}^1} \times \{ 1 \}$ agrees with $\mathfrak{s}_{0,1} \times \mathfrak{s}_{0,1}$ and 
restricted to $\overline{\mathcal{M}}_{\mathcal{T}} \times \{ 0 \}$ agrees with $ \mathfrak{s}_1 $.

From $\tilde{\mathfrak{s}}_1$ we can construct the homotopy $Y_{\mathcal{T}^1}$ required in Definition \ref{gluing-hom}. 
\end{proof}

\section{Linking number and Gromov-Witten invariants}

In this section we assume that \emph{each connected component of $L$ has the rational homology of a sphere}. 
We apply the formalism of the previous sections to define the open Gromov-Witten invariants.
We show that in each equivalence class of system of chains with the gluing property there exists a special element. 
This element is characterized by rational numbers. These are the open Gromov-Witten invariants.

The construction is based on generalized linking numbers. 
We extend the notion of linking number of two curves to a system of chains with gluing property up to homotopy. 

Let us start recalling the definition of linking number between curves on a $3$-sphere suitable for our proposes. 
Let $\gamma_1$ and $\gamma_2$ be two curves on $L$. We will think to $\gamma_1$ and $\gamma_2$ as $1$-chains on $L$. 
The product $\gamma_1 \times \gamma_2$ defines a $2$-chain in $L \times L$ which is homologically trivial because $H_1(L, \Q)=0$. 

Moreover the product structure of $\gamma_1 \times \gamma_2$ distinguishes a particular class of 
$3$-chain $C$ in $[0,1] \times L \times L$ transversal to the boundary such that $\partial C =\{ 0 \} \times \gamma_1 \times \gamma_2$ as follows. 
There exists $B_1, B_2 \in C_2([[0,1]\times L ])$ such that $\partial B_1 = \{ 0 \} \times \gamma_1$ and $\partial B_2 =\{ 0 \} \times \gamma_2$.
Since $H_2(L , \Q)=0$, $B_1$ and $B_2$ are unique up to equivalence.
The homotopies $C_1= B_1 \times \gamma_2$ and $C_2= \gamma_1 \times B_2$ are equivalent. Take $C$ in this equivalence class. 
The linking number between $\gamma_1$ and $\gamma_2$ is defined by the intersection number of $C$ with the diagonal. 

We now generalize the definition of linking number to a system of chains. 
For this we have to consider more general trees, where the case of two curves corresponds to the tree with two vertices joint by an edge.
Moreover we have to deal with singular chains that are not just the product of $1$-chains. 
In the case of two curves the product structure was critical to define uniquely the homotopy. 
In the general case we will use the gluing property up to homotopy.   

Denote by $\mathcal{T}_0$ the set of decorated tree with no edges (and therefore exactly one vertex).
$$ \mathcal{T}_0 = \{ T_0(A)  \}_{A \in H_2(X,L)} $$
where $T_0(A)$ is the decorated tree with no edges and with vertex in the relative homological class $A$.
\begin{proposition} \label{linking}
Let $W_{\mathcal{T}}$ be a system of chains with the gluing property up to homotopy.

There exists a uniquely determined system of chains $  W'_{\mathcal{T}} $ homotopic to $W_{\mathcal{T}}$ such that 
$$ W_T' =0 $$
for every $T \notin \mathcal{T}_0$. 
\end{proposition}
\begin{proof}

We will construct the homotopy using an iterative argument. 
In this process we provide a recipe that identifies uniquely the homotopy up to equivalence. This is essentially based on the gluing property. 

Fix $A \in H_2(X,L)$ and let $l$ be an integer. 
Assume that the homotopy has been constructed for all the trees with symplectic area strictly less than $\omega(A)$ 
or with area equal to $\omega(A)$ and number of external edges less than $l$.

Let $T_{max} \in \mathcal{T}(A)$ be a decorated tree in the homology class $A$ with $l$ external edges and maximal number of internal edges. 
Formula (\ref{boundary-collection}) implies $\partial W_{T_{max}} =0$. 
Using the gluing property first and the fact that $H_1(L,\Q)=0$ after we can construct an homotopy 
$Y_{T_{max}} \in C_*([0,1] \times L_{T_{max}})$ between $W_{T_{max}}$ and the zero singular chain. $Y_{T_{max}}$ uniquely up to equivalence  
since $H_2(L,\Q)=0$ and $Y_{T_{max}}$ has been constructed using the gluing property.  
Lemma \ref{contracting} implies that we can extend $Y_{T_{max}}$ to an homotopy of system of chains. The resulting system of chain has $T_{max}$-component equal to zero.  


We now iterate the argument. Consider a tree $T_n$ with $l$ external edges and suppose that 
$Y_T$ has been constructed for all the trees $T$ with more internal edges than $T_n$ such that (\ref{boundary-homotopy}) holds with $W_T'=0$. 
Assume also that $Y_T$ has been constructed using the recipe from the gluing property as above. 

For each vertex $v \in V(T_n)$, let $\mathcal{T}_v$ be the set of trees with homological class $A_v$ and external edges $H(v)$.
Let $\mathcal{T}_{T_n}$ be the set trees $T$ made from $T_n$ replacing each vertex $v$ of $T_n$ by an element of $\mathcal{T}_v$. That is
$$  \mathcal{T}_{T_n} = (\prod_{v \in V(T_n)} \mathcal{T}_v) / \text{Aut}(T_n)  $$
By induction on each $\mathcal{T}_v$ we have an homotopy to the zero system of chain.
The product of these homotopies defines an homotopy on $\mathcal{T}_{T_n}$ to the zero system of chain.
On the other hand on  $\mathcal{T}_{T_n} \setminus \{ T_n \}$ we have the homotopy contructed before $Y_T$. 
Since both these homotopies have been contructed using the recipe from the gluing property, they have to be equivalent.
This determines an homotopy between $W_{T_n}$ and the zero sistem of chain.

By iterative argument we construct $  W'_{\mathcal{T}} $ as in the proposition. 
The unicity follows since the homotopy $  Y_{\mathcal{T}} $ is determinated up to equivalence.




\end{proof}

\begin{lemma} \label{contracting} Let $W_{\mathcal{T}}$ be as system of chains. 
Assume that for each tree $T $ with $|E(T)| >l$ is assigned a singular chains $ W_T' \in C_{|E(T)|}([0,1] \times L_T , \mathfrak{o}_T )$ and 
$ Y_T \in C_{|E(T)|+1}([0,1] \times L_T , \mathfrak{o}_T )$ such that formula (\ref{boundary-homotopy}) holds.
Then the set $\{ W_T' \}$ can be completed to a system of chains $W_{\mathcal{T}}'$ such that
 the set $\{ Y_T \}$ can be completed to a homotopy $Y_{\mathcal{T}}$ between $W_{\mathcal{T}}$ and $W_{\mathcal{T}}'$. 
Moreover the homotopy $Y_{\mathcal{T}}$ can be taken unique up to equivalence.  
\end{lemma}
\begin{proof}
We construct $Y_{\mathcal{T}}$ by iteration starting from the trees with maximal number of edges with fixed number of external edges. 
Let $T$ be a tree with $l$ internal edges. Define 
$$F_T=\sum_{T'/e'=T}{\partial_{e'} Y_{T'}}$$
where the sum is over all the pairs $(T,e')$ such that $T'/e'=T$. 
Let $W'_T$ be the image of $F_T$ under the projection $[0,1] \times L_T \rightarrow L_T$. 
There exists a singular chain $K_T$ such that 
\begin{equation} \label{K}
\partial K_T = F_T - \{ 1 \} \times W'_T .
\end{equation}
In fact the quivalence class of $K_T$ such that (\ref{K}) holds can be uniquely determinated in the following way.
Write $   F_T = \sum_a \rho_a f_a $ for $\rho_a \in \Q$ and $ f_a= (t_a, g_a) : \Delta \rightarrow [0,1] \times L_T$. 
We put $K_T= \sum_a \rho_a \tilde{f}_a$ where 
$$  \tilde{f}_a : [0,1] \times \Delta \rightarrow [0,1] \times L_T $$
$$  \tilde{f}_a(t,x) = ((1-t) t_a(x) +t, g_a(x))  .$$

Assume also $K_T$ transversal to all the thick diagonals of $L_T$.
Define
$$ Y_T = [0,1] \times W_T + K_T .$$
With this choice formula (\ref{boundary-homotopy}) holds. 
\end{proof}

Proposition \ref{linking} allows us to define the linking number of $W_{\mathcal{T}(A)}$ by
$$\text{link}(W_{\mathcal{T}(A)}) =  W_{T_0(A)}' \in \Q .$$
Observe that $L_{T_0(A)} $ is just a point. 

\begin{theorem} Let $W_{\mathcal{T}}$ the system of chains (\ref{chain}). 
The rational number $\text{link}(W_{\mathcal{T}(A)})$ is defined and it does not depend on the almost complex structure and various choices we made to define a Kuranishi structure.  
\end{theorem}
\begin{proof}
By Lemma \ref{gluing-lemma}, $W_{\mathcal{T}}$ has the gluing property up to homotopy. 
Therefore we can apply Proposition \ref{linking} to define $\text{link}(W_{\mathcal{T}(A)})$.
Theorem \ref{invariance} implies that this rational number does not depend on any choice we made.
\end{proof}
We can now define the Open Gromov-Witten invariants that count multi-disks in the relative homology class $A \in H_2(X,L)$ as
$$ F_{0,1}(A) = \text{link}(W_{\mathcal{T}(A)}) $$

It also easy to see how $F_{0,1}(A)$ depends on the singular chain $B$ that we have used to define (\ref{alternative}). Assume that $B' \in C_4(M)$ is another singular chain with $\partial B'= L$. Since $B - B'$ is a cycle we have
\begin{equation} \label{change}
F_{0,1}(A) - F_{0,1}'(A) = \text{ev}_* ([\overline{\mathcal{M}}_{0,1} (A)]) \cap (B - B') .
\end{equation}

\section{Higher genus}

In this section we generalize the previous analysis to higher-genus curves. 
Most of what we said for disks can be extended to higher-genus straightforwardly. Therefore we will only comment to the main differences. 

In this section we \emph{assume that $L$ is spin}. This assumption assures that all the moduli spaces we consider are oriented (see \cite{L}). 

\subsection{Multi-curves}

Let $\mathcal{M}_{(g,h),(n, \overrightarrow{m})}(A)$ be the Kuranishi space of stable maps of type $(g, h)$ with $n$ internal marked points and 
$ \overrightarrow{m}=(m_1,...,m_h)$ boundary marked points, representing the relative homology class $A \in H_2(M,L)$. This space has been studied in \cite{L}.

\begin{definition}

A decorated graph $G$ is a ribbon graph endowed with the following data

\begin{itemize}
\item A set $I(G)$. For each $i \in I(G)$ a relative homological classes $A_i \in H_2(M,L)$ and an oriented surface $ \Sigma_i$ with no internal marked points. 
Let $g_i$ be the genus, $h_i$ the number of boundary components and $\overrightarrow{m_i}$ be the number of boundary marked points of $\Sigma_i$. 
If $A_i=0$ we assume that $\Sigma_i$ is stable.
\item A one to one correspondence between vertices of $G$ and the boundary components of $ \{ \Sigma_i \}_{i \in I}$.
\item For each $v \in V(G)$, a one to one correspondence between the half edges $H(v)$ starting in $v$ and the boundary marked points of the boundary component associated to $v$.
\end{itemize}

The homology class of $G$ is given by $A = \sum_{i \in I} A_i$.
From a decorated graph $G$ we can construct a surface $\Sigma_G$ replacing each edge of $G$ with a strip respecting the orientation of the surfaces $\Sigma_i$. 
We say that the decorated graph $G$ is connected if the surfaces $\Sigma_G$ is connected. 
\end{definition}

Denote by $\mathcal{G}_{(g,h),\overrightarrow{m}}(A)$ the set of connected decorated graphs $G$ where $\Sigma_G$ has genus $g$, $h$ boundary components, 
$\overrightarrow{m}$ external edges, representing the relative homology class $A \in H_2(M,L)$.

For each $i \in I$ let $\overline{\mathcal{M}}_i$ be a copy of $ \overline{\mathcal{M}}_{(g_i, h_i), (0, \overrightarrow{m}_i)} (A_i)$.

The moduli space of multi-curves $\overline{\mathcal{M}}_G$ is defined as: 
$$ \overline{\mathcal{M}}_G =  \left( \prod_{i \in I} \overline{\mathcal{M}}_i \right) / \text{Aut}(G) . $$

Since we assume that $L$ is spin, every $\overline{\mathcal{M}}_i$ has a natural orienation. 
This implies that $ \overline{\mathcal{M}}_G $ has a natural orientation with twisted coefficients in $\mathfrak{o}_G$.

Let $L_G$ be the orbifold
$$ L_G = L^{H(G)}/ \text{Aut}(G) .$$

The evaluation map on the punctures defines a natural map 
\begin{equation} \label{higher-ev}
\text{ev} : \overline{\mathcal{M}}_G \rightarrow L_G 
\end{equation}

\subsection{Boundary}

In \cite{L} is proved that the boundary $ \overline{\mathcal{M}}_{(g, h), (n, \overrightarrow{m})}$ can be subdivided in components of type $E$, $H1$, $H2$, $H3$. 

Let $G$ be a decorated graph. 

The boundary faces of $\overline{\mathcal{M}}_G$ correspond to the boundary faces of the component surfaces $I(G)$. 
As in the case of multi-disks, we want to identify these boundary faces with subspaces associated to other graphs. 
 
For each internal edge $e \in E(G)$ define
$$ \partial_e \overline{\mathcal{M}}_G = ( \text{ev}_e^{-1}(\Delta)) /  \text{Aut}(G,e) $$
where $\text{ev}_e : \prod_{i \in I(G)} \overline{\mathcal{M}}_i \rightarrow  L \times L $ is the evaluation map associated to the edge $e$.

\begin{definition} \label{higher-contracting}
For $e \in E(G)$ an internal edge of $G$.
Let $v$, $w$ be the vertices attached to $e$. Let $i_v, i_w \in I(G)$ be the associated surfaces with boundary components $v$ and $w$ respectively.   
The decorated graph $G/e$ is defined replacing the surfaces $\Sigma_{i_v}$ and $\Sigma_{i_w}$ by the surface made by gluing the boundary marked points associated to $e$ and smoothing the resulting node.
\end{definition}

Observe that in the preview definition we can have $i_v=i_w $ or even $v =w$. The boundary node that we smooth is of type  
\begin{itemize}
\item H1 if $v= w$,
\item H2 if $v \neq w$ but $i_v = i_w$,
\item H3 if $i_v \neq i_w$ .
\end{itemize}



\begin{definition} \label{higher-contracting0}
Let $V_0(G)$ be the set of vertices of $G$ with valence zero. Each $v \in V_0(G)$ correspond to a boundary component of some $\Sigma_{i_v}$ ($i \in I(G)$) without marked points. 
Let $\Sigma_{i_v'}$ be the surface that we get from $\Sigma_{i_v}$ shrinking the boundary component associated to $v$ to an internal marked point $*$.
Let $ \overline{\mathcal{M}}_{G/v} $ be the associated moduli space of multi-cuves:
$$ \overline{\mathcal{M}}_{G/v} =  \left(  \overline{\mathcal{M}}_{i_v'} \times \prod_{i \neq i_v} \overline{\mathcal{M}}_{i}  \right) / \text{Aut}(G,v) $$

\end{definition}



As in the case of multi disks, the Kuranishi spaces $ \overline{\mathcal{M}}_G$ as a natural orientation with twisted coefficients $\mathfrak{o}_G)$.  
With this notation we can formulate the generalization of Lemma \ref{attach-boundary}:  
\begin{lemma} 
For each decorated graph $G$ the following is an isomorphism of oriented Kuranishi spaces: 
\begin{equation} \label{higher-attach}
\partial \overline{\mathcal{M}}_G   \cong (\bigsqcup_{G'/e'=G} \partial_{e'} \overline{\mathcal{M}}_{G'}) \sqcup ( \bigsqcup_{v \in V_0(G)}  \overline{\mathcal{M}}_{G/v} \times_{ev_*} L). 
\end{equation} 
\end{lemma}


\subsection{Perturbations and Systems of chains}

The definition of system of chains and homotopy between system of chains extends straightforward to decorated graphs.

Assume for the moment that the last term of (\ref{higher-attach}) is absent. See section \ref{E} for the necessary modifications.

We consider pertubations $\mathfrak{s}$ of the Kuranishi structure of 
$$ \bigsqcup_{G \in \mathcal{G}} \overline{\mathcal{M}}_G $$ 
with the property
\begin{equation} \label{higher-condition}
\mathfrak{s}|_{\partial_e \overline{\mathcal{M}}_G} = \bigsqcup_{G'/e'=G} \mathfrak{s}|_{\partial_{e'} \overline{\mathcal{M}}_{G'}}.
\end{equation}

As in Lemma \ref{trasverse}, we can prove that there exists a transversal perturbation $\mathfrak{s}$ as close as we want to $s$ in the $C^0$ topology, such that (\ref{higher-condition}) holds. 
For each such $\mathfrak{s}$ we define
\begin{equation} \label{higher-chain}
W_G =  \text{ev}_* ((\mathfrak{s}|_{\overline{\mathcal{M}}_G}^{-1})(0)) \in C_*(L_G , \mathfrak{o}_G) .
\end{equation}
where $\text{ev}$ is defined in (\ref{higher-ev}).

As in the case of multi-disks we have (remember that here we neglect the last term of (\ref{higher-attach})): 
\begin{proposition} \label{higher-invariance}
$ \{ W_G \}_{G \in \mathcal{G}} $ defines a system of chains.
Different choices of perturbations satisfying condition (\ref{higher-condition}) lead to systems of chains homotopic, with homotopy uniquely determined up to equivalence.
\end{proposition}

\subsection{Boundary nodes of type E} \label{E}

We now modify the formulas above in order to deal with the last term of (\ref{higher-attach}). 
We assume that there exists a singular chain $B \in C_4(X)$ such that $\partial B =L$.

Let $E \subset V_0(G)$ be a subset of $V_0(G)$. For each $i \in I(G)$, denote by $E_i$ the intersection of $E$ with the vertices associated to boundary components of $\Sigma_i$ 
Let $\Sigma_i^{E_i} $ be the surface made by $\Sigma_i$ contracting the boundary components associated to $E_i$ to an internal marked point. 
Let  $\overline{\mathcal{M}}_i^{E_i}$ be the associated Kuranishi space.
The evaluation on the internal marked point gives a map 
\begin{equation} \label{higher-ev2}
\prod_{i \in I(G)} \overline{\mathcal{M}}_i^{E_i} \rightarrow   M^E .
\end{equation}
Define
$$  \overline{\mathcal{M}}_{G,E} =   \left( \left( \prod_{i \in I} \overline{\mathcal{M}}_i \right) \times_{M^E} \left(  B^E \right)  \right) / \text{Aut}(G) . $$
\begin{equation} \label{strong2}
\text{ev}:\overline{\mathcal{M}}_{G,E}(J) \rightarrow L_G .
\end{equation}
As in formula $(6.10)$ of \cite{FO}, the strongly continuous map (\ref{strong2}) defines a singular chain   
\begin{equation} \label{chain2}
W_G = \sum_{E} \text{ev}_* ((\mathfrak{s}|_{\overline{\mathcal{M}}_{G,E}}^{-1})(0))
\end{equation}

\subsection{Gluing property} 
Let $\mathcal{G}^{k,l}$ the set of graphs with $k$ ordered internal edges and $l$ ordered external edges. We also denote $\mathcal{G}^k=\mathcal{G}^{k,0}$.
The cut of the internal edge gives an isomorphism
$$  \mathcal{G}^1 = (  \mathcal{G}^{0,2} \sqcup \mathcal{G}^{0,1} \times \mathcal{G}^{0,1} ) / \Z_2   $$

$W_{\mathcal{G}}$ has the the gluing property up to homotopy if it is assigned an equivalence class of homotopies $W_{\mathcal{G}^1}$ between 
$W_{\mathcal{G}^1}$ and $W_{\mathcal{G}^{0,2}} \sqcup W_{\mathcal{G}^{0,1}} \times W_{\mathcal{G}^{0,1}}$ 
such that the homotopy between $W_{\mathcal{G}^2}$ and 
$W_{\mathcal{G}^{0,4}} \sqcup W_{\mathcal{G}^{0,1}} \times W_{\mathcal{G}^{0,3}} \times W_{\mathcal{G}^{0,1}} \sqcup W_{\mathcal{G}^{0,1}} \times W_{\mathcal{G}^{0,2}} \times W_{\mathcal{G}^{0,1}}$
\begin{equation} \label{YY-comp-higher}
Y_{\mathcal{G}^2} \circ (  Y_{\mathcal{G}^{1,2}}  \sqcup W_{\mathcal{G}^{0,1}} \times Y_{\mathcal{G}^{1,1}})
\end{equation}
is invariant (up to equivalence) by the $\Z_2$ action that switch of the order of the marked edges.


\begin{proposition} \label{higher-gluing-lemma}
The system of chain defined in (\ref{higher-chain}) has the gluing property up to homotopy.
\end{proposition}

\subsection{Linking numbers}
Assume that \emph{each connected component of $L$ has the rational homology of a sphere}. As in the case of disks we have: 
Denote by $\mathcal{G}_0$ the set of connected decorated graphs with no edges (and therefore exactly one surfaces component).
$$ \mathcal{G}_0 = \{ G_{g,h}(A)  \}_{A \in H_2(X,L)} $$
where $G_{g,h}(A)$ is the decorated graph with no edges and with $I(G_{g,h}(A))$ is given by a unique surface of genus $g$, $h$ boundary components and relative homological class $A$.
\begin{lemma} \label{higher-linking}
Let $W_{\mathcal{G}}$ be a system of chains with the gluing property up to homotopy.

There exists an uniquely determined system of chains $  W'_{\mathcal{G}} $ homotopic to $W_{\mathcal{G}}$ such that 
$$ W_G' =0 $$
for every $G \notin \mathcal{G}_0$. 
\end{lemma}

Define the Gromov-Witten invariant as the rational number 
$$ F_{g,h}(A)= W_{G_{g,h}(A)}' \in \Q  .$$
From Lemma \ref{higher-invariance} and Lemma \ref{higher-linking} follows:
\begin{theorem}
The rational numbers $F_{g,h}(A)$ do not dependent on the almost complex structure and various choices we made to define the Kuranishi structure. 
\end{theorem}

Observe that the rational numbers $F_{g,h}(A)$ depend not only on the pair $(X,L)$ by also on the choice of the singular chain $B$ in section \ref{E}. 
However, for different choices of $B$ the invariants can be expressed by a formula analogous to (\ref{change}).   




\end{document}